\documentclass[12pt]{amsart}

\usepackage{setspace}
\usepackage{verbatim}

\usepackage{amsfonts,fancyhdr,ifthen,bm}
\usepackage{amscd,amssymb} 
\usepackage{times}
\usepackage{graphicx}
\usepackage{color}
\usepackage{epsfig}
\usepackage{mathrsfs}
\usepackage{paralist}
\usepackage{comment}
\usepackage{mathtools}
\usepackage{multirow}

\newtheorem{thm}{Theorem}[section]
\newtheorem*{thm*}{Theorem}
\newtheorem{prop}[thm]{Proposition}
\newtheorem*{prop*}{Proposition}

\newtheorem*{cor*}{Corollary}

\newtheorem{lem}[thm]{Lemma}
\newtheorem*{lem*}{Lemma}

\newtheorem*{oquest*}{Open Question}

\theoremstyle{remark}

\theoremstyle{remark}
\newtheorem*{rmk*}{Remark}

\theoremstyle{definition}
\newtheorem{defn}[thm]{Definition}

\theoremstyle{definition}
\newtheorem{ass}[thm]{Assumption}

\theoremstyle{definition}
\newtheorem*{defn*}{Definition}

\theoremstyle{definition}

\theoremstyle{definition}
\newtheorem{exc}[thm]{Exercise}

\numberwithin{equation}{section}

\newcommand{\Z}{\mathbb{Z}}
\newcommand{\QQ}{\mathbb{Q}}

\newcommand{\SelT}{\text{Sel}^{(2)}}
\newcommand{\QmodSq}{\mathbb{Q}^{\times}\big/(\mathbb{Q}^{\times})^2}

\DeclareFontFamily{U}{wncy}{}
\DeclareFontShape{U}{wncy}{m}{n}{<->wncyr10}{}
\DeclareSymbolFont{mcy}{U}{wncy}{m}{n}
\DeclareMathSymbol{\Sha}{\mathord}{mcy}{"58}

\begin{document}


\title[4-Selmer groups and 8-class groups]{
Governing fields and statistics for 4-Selmer groups and 8-class groups
}

\author{
Alexander Smith
}

\thanks{I would like to thank Bjorn Poonen for his useful feedback during the final stages of this project.}

\date{\today}

\begin{abstract}
Taking $A$ to be an abelian variety with full $2$-torsion over a number field $k$, we investigate how the $4$-Selmer rank of the quadratic twist $A^{(d)}$ changes with $d \in k^{\times}$. We show that this rank depends on the splitting behavior of the primes dividing $d$ in a certain number field $L/k$.

Assuming the grand Riemann hypothesis, we then prove that, given an elliptic curve $E/\QQ$ with full rational $2$-torsion, the quadratic twist family of $E$ usually has the distribution of $4$-Selmer groups predicted by Delaunay's heuristic. Analogously, and still subject to the grand Riemann hypothesis, we prove that the set of quadratic imaginary fields has the distribution of $8$-class groups predicted by the Cohen-Lenstra heuristic.
\end{abstract}

\maketitle

\section{Introduction}
\label{sec:intro}
In \cite{CoLe84}, Cohen and Lenstra gave a heuristic for understanding the distribution of class groups of number fields. Writing $K_d = \QQ(\sqrt{d})$ and taking $H$ to be any finite abelian $p$-group for $p \ne 2$, their work predicts that the probability
\[\lim_{N \rightarrow \infty} \frac{1}{N}|\{-N < d < 0 \,: \, \text{Cl}\, K_d[p^{\infty}] \cong  H \}|\]
exists, is positive, and is inversely proportional to the number of automorphisms of $H$.  In \cite{FrMa89}, Friedman and Washington proved that this distribution of groups is approached by the limit of the distribution of cokernels of random $m \times m$ matrices with entries in $\Z_p$ as $m$ increases.

It is generally accepted that Tate-Shafarevich groups of abelian varieties are analogous objects to class groups of number fields, and this analogy carries over to heuristics.  The direct elliptic curve analogy of the Cohen-Lenstra heuristic was first found by Delaunay \cite{Del01}, who gave concrete conjectures for the distribution of the groups $\Sha(E)[p^{\infty}]$ over the set of elliptic curves $E/\QQ$. Delaunay's heuristic was placed in a more general setup by Bhargava, Kane, Lenstra, Poonen, and Rains \cite{BKLPR15}. These five authors suggested that the distribution of $p^{\infty}$-Selmer groups of rank $r = 0, 1$ elliptic curves is approached by the distribution of cokernels of random alternating $(2m + r)$-dimensional matrices with entries in $\Z_p$ , and they proved that this heuristic was consistent with that of Delaunay. We will refer to this model as the BKLPR heuristic.

A glaring difference between the Cohen-Lenstra and BKLPR heuristics is that the former is restricted to $p \ne 2$, while the latter has no such restriction. This is easy enough to explain. Per Gauss's genus theory, the two torsion of the class group $\text{Cl}\, K_d$ of an imaginary quadratic field is represented by all ideals with squarefree norm dividing the discriminant $\Delta$ of $K_d$. The only relations among these ideals are the ones generated by
\[(1) \sim (\sqrt{d}).\]
Then, if $\Delta$ is divisible by exactly $r$ distinct prime factors, we get $\text{Cl}\, K_d[2] \cong (\Z/2\Z)^{r-1}$. Because of this, over the family of negative $d$, the probability that $\text{Cl}\,K_d[2^{\infty}]$ is isomorphic to $H$ is zero for all $2$-groups $H$. Such an issue is not present for Selmer groups, so the restriction $p \ne 2$ is only needed for the Cohen-Lenstra heuristics.

As reported by Wood \cite{Wood14}, Gerth found a workaround for the Cohen-Lenstra heuristics at $p = 2$ in 1987. Since the two torsion of the class group was uninteresting, Gerth considered instead how the $2$-Sylow subgroup of $2\text{Cl} \, K_d$ varied for $d$ in a family, where we are using additive notation for the class group. He predicted that, if $H$ is a finite abelian $2$-group, then 
\[\lim_{N \rightarrow \infty} \frac{1}{N}|\{-N < d < 0 \,: \, 2(\text{Cl}\, K_d[2^{\infty}] )\cong  H \}|\]
exists, is positive, and is inversely proportional to the number of automorphisms of $H$.

Even with this addendum, $2$ has remained an outsider prime in articles about the Cohen-Lenstra heuristics. This is particularly obvious from articles whose main results are stated for odd $p$ but whose proofs apply equally well for $2(\text{Cl}\, K[2^{\infty}])$ and $\text{Cl}\,K[p^{\infty}]$; this is true of the Friedman-Washington result, for example. Though less put upon than the corresponding class group, the $2^{\infty}$-Selmer group is also in a category apart from the other $p^{\infty}$-Selmer groups. At least they have each other; for any $k \ge 2$, there is a particularly close correspondence between the  groups
\[ 2(\text{Cl} \,K_{\Delta}[2^{k}])\]
over the set of negative quadratic discriminants $\Delta$ and
\[\text{Sel}^{(2^{k-1})}(E^{(d)})\big/ E^{(d)}[2]\]
over the set of quadratic twists $E^{(d)}$ of some elliptic curve $E/\QQ$ with full rational $2$-torsion, with this quotient coming from the map
\[E^{(d)}[2] \xrightarrow{\quad\,\,\,} E^{(d)}(\QQ)/2E^{(d)}(\QQ) \xhookrightarrow{\quad\,\,\,} \text{Sel}^{(2^{k-1})}(E^{(d)}).\]

The reason that $p = 2$ is such an outsider is again easy to explain. Even though $2(\text{Cl}\, K_d[2^{\infty}])$ no longer has its shape dictated by Gauss genus theory, it is still affected by it. Write $\Delta = d_2p$ and $\Delta' = d_2p'$ for the discriminant of two imaginary quadratic fields, where $p, p'$ are odd primes. Write $V_{\Delta}$ for the subspace of $\QmodSq$ generated by positive divisors of $d_2$, with $V_{\Delta'}$ the analogous space for $\Delta'$. Then genus theory quickly gives us
\[\text{Cl}\, K_{\Delta}[2] \,\cong V_{\Delta} = V_{\Delta'} \cong\, \text{Cl}\, K_{\Delta'}[2].\]
In other words, the two torsion of the class groups of $K_{\Delta}$ and $K_{\Delta'}$ can be given identical arithmetic structure. Furthermore, supposing $pp'$ is a square mod $8d_2$, we find that the image of $2(\text{Cl}\, K_{\Delta}[4])$ in $V_{\Delta}$ equals that of $2(\text{Cl}\, K_{\Delta'}[4])$ in $V_{\Delta'}$. With this, the identical structure for $2$-class groups is passed on to an identical class structure for $4$-class groups. In general, for any $k > 1$, we can partition the set of all primes not dividing $2d_2$ into a finite number of sets so that
\[2^{k-1}(\text{Cl}\, K_{\Delta}[2^k]) = 2^{k-1}(\text{Cl}\, K_{\Delta'}[2^k])\]
inside $V_{\Delta}$ whenever $p$ and $p'$ come from the same set. Analogously, for any $k \ge 2$, we can split the set of primes not dividing twice the conductor of $E^{(d_2)}/\QQ$ into finitely many classes so that the arithmetic structure of
\[2^{k-2} \big(\text{Sel}^{(2^{k-1})}(E^{(d_2p)})\big/ E^{(d_2p)}[2]\big)\]
only depends on the class of $p$. There is no analogue for this arithmetic invariance for any odd class group or any odd Selmer group.

In the case $k = 2$, the advantages of the simple structure of $2^{k-1}$-Selmer groups and $2^{k}$-class groups have been well used. On the class group side, Fouvry and Kl{\"u}ners  proved that the distribution of $4$-class ranks of quadratic fields are consistent with Gerth's heuristic \cite{Fouv07}. On the other side, Kane \cite{Kane13} and Swinnerton-Dyer \cite{Swin08} have proved that, if $E/\QQ$ is an elliptic curve with full $2$-torsion and no cyclic subgroup of order $4$ defined over $\QQ$, the distribution of $\SelT(E^{(d)})/E[2]$ in the quadratic twist family of $E$ is as predicted by the BKLPR heuristics. These results also show the interplay between $4$-class groups and $2$-Selmer groups, with the Fouvry-Kl{\"u}ners result partially based on earlier methods coming from the analysis of the $2$-Selmer group of the congruent number curve \cite{Heat94}.

For $k = 3$, progress has been concentrated on the side of $2^k$-class groups. In 1983, Cohn and Lagarias conjectured that the $8$-class rank of $\QQ(\sqrt{dp})$ was determined by the Artin class of $p$ over some finite extension $L/\QQ$ determined by $d$. They called such extensions \emph{governing fields} \cite{CoLa83}. Stevenhagen gave the first proof that such governing fields existed \cite{Stev89}, and many other authors \cite{Cohn69, Kapl73, Mort82, Mort90, WuYue07, JuYue11, Lu15} found results in special cases.

Unbeknownst to these authors, substantial results towards the construction of governing fields of $8$-class groups had been found by R{\'e}dei in 1939 \cite{Rede39}. Using this, minimal governing fields were found in all cases by Corsman in his thesis \cite{Cors07} using a variant of the argument outlined in Exercise \ref{exc:okay_then}. However, despite the redundancy, the work of Cohn and Lagarias remains important, as governing fields remain an interesting way of understanding the structure of $4$-Selmer groups.

Essentially the only prior evidence for $4$-Selmer groups being regulated by a governing field comes from Hemenway's thesis \cite{Heme06}, where a splitting condition characterizing some non-congruent primes was derived by calculating $4$-Selmer groups.  However, as we show in Theorem \ref{thm:4sel_gov}, governing fields of $4$-Selmer groups are quite ubiquitous; indeed, there are governing fields for arbitrary abelian varieties over number fields $A/k$ whenever $A[2]$ has full $2$-torsion over $k$.

The goal of this paper is thus established. We will conditionally extend the results of Fouvry-Kl{\"u}ners to $8$-class groups and the results of Kane and Swinnerton-Dyer to $4$-Selmer groups. To do this, we will use the idea of governing fields, which are well established for $8$-class groups and not established for $4$-Selmer groups. Section \ref{sec:8class} begins the study of $8$-class groups, rederiving minimal governing fields and finding genericity conditions under which the governing field is particularly nice. 

Section \ref{sec:4sel} begins by finding governing fields for the $4$-Selmer group of abelian varieties over number fields. Specializing to elliptic curves over $\QQ$, we again find genericity conditions under which the governing field is particularly nice.

One of the nice properties coming from the genericitiy conditions is that, in generic families of discriminants indexed by one prime $p$, or in generic families of twists indexed by one prime $p$, the $8$-class rank and $4$-Selmer rank are distributed as expected by the Cohen-Lenstra and BKLPR heuristics respectively. In the final section of this paper, we bundle the distributions found in these small families together, finding the distribution of $8$-class groups for imaginary quadratic fields and of $4$-Selmer groups for quadratic twist families. To do this bundling, we need a strong, effective form of the Chebotarev density theorem, so we assume some cases of the grand Riemann hypothesis in these final theorems.

This work leaves two big questions unanswered. The first is whether the grand Riemann hypothesis is really central to these results. I hazard that it is not. A large sieve can be used to prove an averaged form of the Riemann hypothesis for Dirichlet $L$-functions \cite[Ch. 7]{Iwan04}, and I would guess that a similar average indexed over Artin $L$-functions of certain $2$-group extensions of $\QQ$ would be powerful enough to give the distributions we prove conditionally. However, I have not yet seen such a large sieve.

The second question is whether this work can be extended to higher Selmer groups or class groups. The analogy between the two situations will doubtlessly still be useful, although the evidence of this thus far is limited to one paper \cite{BrHe13}. Even if governing fields in the sense of this paper cannot be found, there will likely still be interesting structure. For example, in some situations, we expect the $8$-Selmer structure of $E^{(d_1d_2)}$ to be partially determined by the structures of $E$, $E^{(d_1)}$, and $E^{(d_2)}$; for the $16$-Selmer structure of $E^{(d_1d_2d_3)}$ to be partially determined by $E$, the $E^{(d_i)}$, and the $E^{(d_id_j)}$ for $i, j$ in $\{1, 2, 3\}$; etc. If these correspond to relations among the corresponding $L$-functions, they might give a new way of understanding how these objects are related. This will need to wait for a later paper.

\section{Governing fields for 8-class groups}
\label{sec:8class}
We start with a reciprocity law originally due to R{\'e}dei \cite{Rede39}. For $a, b$ nonsquare integers, let $K_{a,b}$ be the field $\QQ(\sqrt{a}, \sqrt{b})$. This field is of degree four over $\QQ$ unless $ab$ is a square. Take $(\,\, ,\,\,)_v$ to be the Hilbert symbol. If $(a, b)_v = +1$ at all rational places, the equation
\[x^2 - ay^2 = bz^2\]
has a solution in rationals. Choose $x, y, z$ so that 
\[L_{a,b} = K_{a,b}\left(\sqrt{x + y \sqrt{a}}\right).\]
is unramified above $K_{a, b}$ at all primes not ramified in both $\QQ(\sqrt{a})/\QQ$ and $\QQ(\sqrt{b})/\QQ$. This field is quadratic above $K_{a,b}$ and is either a $D_8$ or $\Z/4\Z$ Galois extension of $\QQ$. The above process can produce many potential fields $L_{a,b}$, but this will not matter.

\begin{prop}
\label{prop:dih_rec}
Let $a$, $b$, and $c$ be squarefree rational integers not equal to $1$. We assume that $c$ is positive, equal to $1$ mod $8$, and relatively prime to $a$ and $b$. Also assume that we have
\[(a, b)_v = (a, c)_v = (b, c)_v = +1\]
at all rational places $v$.

Following the method above, we find fields $L_{a,b}/K_{a,b}$ and $L_{a,c}/K_{a,c}$. Also, we can find an ideal $\mathbf{b}$ of $K_{a,c}$ that has norm $|b|$ in $\QQ$, and an ideal $\mathbf{c}$ of $K_{a,b}$ that has norm $c$ in $\QQ$.

Identify the Galois groups of $L_{a,c}/K_{a,c}$ and $L_{a, b}/K_{a,b}$. Then we have the following equality of Artin symbols:
\begin{equation}
\label{eq:dih_rec}
\left[\frac{L_{a,b}/K_{a,b}}{\mathbf{c}} \right] = \left[\frac{L_{a,c}/K_{a,c}}{\mathbf{b}} \right].
\end{equation}
\end{prop}

\begin{proof}

Consider the dihedral group
\[D_8 = \langle r, s\,\, |\,\, r^4 = s^2 = (rs)^2 = 1\rangle.\]
Define a map $\gamma: D_8 \rightarrow \pm 1$ by
\[\gamma(1) = \gamma(r) = \gamma(s) = \gamma(rs) = +1 \,\,\,\,\, \text{and}\]
\[\gamma(r^2) = \gamma(r^3) = \gamma(r^2s) = \gamma(r^3s) = -1.\]
Note that $\gamma(r^2 g) = -\gamma(g)$ for all $g \in D_8$. Then, for all $\sigma$ and $\tau$ in the group, we have that the coboundary $d\gamma(\sigma, \tau) = \gamma(\sigma \tau)\cdot \gamma(\sigma)\cdot\gamma(\tau)$ satisfies
\[d\gamma(r^2\sigma, \tau) = d\gamma(\sigma, r^2\tau) = d\gamma(\sigma, \tau).\]
This together with the values in Table \ref{tab:cobnd_vals} completely determines $d\gamma$. In particular, taking $\chi_1: D_8 \rightarrow \pm 1$ to be the homomorphism with kernel $\langle r^2, rs \rangle$, and taking $\chi_2$ to be the homomorphism with kernel $\langle r^2, s \rangle$, we have
\[d\gamma = \chi_1 \cup \chi_2\]
where the cup product of cocycles is induced by the natural bilinear map $\Z/2\Z \times \Z/2\Z \rightarrow \Z/2\Z$.

For $m$ a nonsquare integer, take $\chi_m$ to be the quadratic character corresponding to $\QQ(\sqrt{m})/\QQ$.
Given $m, n$ squarefree integers not equal to one that satisfy $(m, n)_v = +1$ at all places of $\QQ$, we choose $L_{m, n}$ as before. Writing $G_{\QQ}$ for $\text{Gal}(\bar{\QQ}/\QQ)$, we can choose an injection of this group into $D_8$ and define a map
\[\gamma_{m, n} : G_{\QQ} \xrightarrow{\ \ \ \ } \text{Gal}(L_{m, n}/ \QQ) \xhookrightarrow{\quad \ \ } D_8 \xrightarrow{\ \ \gamma\ \ } \pm 1\]
which satisfies $d\gamma_{m, n} = \chi_m \cup \chi_n$. This applies whether $L_{m, n}$ is a $D_8$ or $\Z/4\Z$ extension of $\QQ$.

\begin{table}
\begin{center}
\begin{tabular}{c  c | c  c  c  c|} \cline{3-6}
 & &  \multicolumn{4}{|c|}{$\tau$} \\
 \cline{3-6}
 & & 1 & $r$ & $s$ & $rs$ \\ \cline{1-6}
 \multicolumn{1}{ |c|}{\multirow{4}{*}{$\sigma$}} & 1 & 1 & 1 & 1 & 1 \\ 
 \multicolumn{1}{ |c|  }{} & $r$ & 1 & -1 & 1 & -1 \\
 \multicolumn{1}{ |c|  }{} & $s$ & 1 & -1 & 1 & -1 \\ 
 \multicolumn{1}{ |c|  }{} & $rs$ & 1 & 1 & 1 & 1 \\ \cline{1-6} \end{tabular}
 \end{center}
 \caption{Values of $d\gamma(\sigma, \tau)$}
 \label{tab:cobnd_vals}
\end{table}

With this notation,
\[\gamma_{c, a} \cup \chi_b\]
and
\[\chi_c \cup \gamma_{a, b}\]
are both well defined Galois $2$-cochains. The coboundary of both of these is
\[\chi_c \cup \chi_a \cup \chi_b.\]
In particular, the coboundary of the difference of these cochains is zero, so the difference is a $2$-cocycle corresponding to an element in $H^2(G_{\QQ}, \pm 1)$. Taking 
\[\text{inv}_v: H^2(G_{\QQ_v}, \overline{\QQ}_v^{\times}) \rightarrow \QQ/\Z\]
to be the canonical map at every place of $\QQ$, we then get
\[\sum_v \text{inv}_v\left( \gamma_{c,a} \cup \chi_b \,-\, \chi_{c} \cup \gamma_{a, b} \right) = 0\]
from class field theory.

We see that this cocycle is unramified at all places outside $2abc\infty$. Then
\[\sum_{v| 2abc\infty} \text{inv}_v\left( \gamma_{c,a} \cup \chi_b \,-\, \chi_{c} \cup \gamma_{a, b} \right) = 0.\]
As $\chi_c$ is locally trivial at all places dividing $2ab\infty$ and $\chi_b$ is locally trivial at all primes dividing $c$, this equation becomes
\begin{equation}
\label{eq:dih_rec_raw}
\sum_{v|c} \text{inv}_v\left(\chi_{c} \cup \gamma_{a, b} \right) = \sum_{v|2ab\infty} \text{inv}_v \left( \gamma_{c, a} \cup \chi_b\right).
\end{equation}
For primes $p$ dividing $c$, $\QQ_p$ will contain $K_{a,b}$ and $L_{a,b}$ will be unramified above $K_{a, b}$. Because of this, $\gamma_{a,b}$ locally becomes either zero or an unramified quadratic character. Then, from the basic properties of $\text{inv}_v$ (see \cite[Ch. XIV]{Serre79}), we have
\[\text{inv}_p\left(\chi_{c} \cup \gamma_{a, b} \right) = \begin{cases} 0 &\text{if } p \text{ splits completely in }L_{a,b}/\QQ\\\frac{1}{2} &\text{otherwise}.\end{cases}\]
Identifying the Galois group of $L_{a,b}/K_{a,b}$ with $\frac{1}{2}\Z/\Z$, we thus have
\[\left[\frac{L_{a,b}/K_{a,b}}{\mathbf{c}} \right]  = \sum_{v|c} \text{inv}_v\left(\chi_{c} \cup \gamma_{a, b} \right).\]

We now look at the other side of \eqref{eq:dih_rec_raw}. At primes dividing $a$ or $c$ but not dividing $2b$, we see $\chi_b$ is locally trivial, so the corresponding invariant will not contribute. So consider a prime $p$ dividing $2b$, and take $t$ to be any integer so $\QQ_p(\sqrt{t})/\QQ_p$ is nonsplit and unramified. Since $c$ is a square mod $p$, we see that $L_{a, c}/K_{a,c}$ is unramified at $p$ and that $\gamma_{a, c}$ is locally a quadratic character. More specifically,
\begin{itemize}
\item If $L_{a,c}/K_{a,c}$ is split, $\gamma_{a, c}$ is locally trivial or is locally equal to $\chi_a$.
\item Otherwise, $\gamma_{a, c}$ is locally equal to either $\chi_t$ or $\chi_{at}$.
\end{itemize}
But we have $\text{inv}_p(\chi_a \cup \chi_b) = 0$ everywhere. Then the same reasoning as before gives
\begin{equation}
\label{eq:rhs_dih_rec}
\left[\frac{L_{a,c}/K_{a,c}}{\mathbf{b}} \right]  = \sum_{v|2ab\infty} \text{inv}_v\left(\gamma_{c, a} \cup \chi_b\right).
\end{equation}
This gives the proposition.
\end{proof}

\begin{exc}
\label{exc:okay_then} \cite{Cors07}
Take $\sqrt{x_b + y_b \sqrt{a}}$ to generate $L_{a, b}$ above $K_{a, b}$, and take $\sqrt{x_c + y_c \sqrt{a}}$ to generate $L_{a, c}$ above $K_{a, c}$. Hilbert reciprocity gives that
\[\prod_v \big(x_b + y_b \sqrt{a}, \,\, x_c + y_c \sqrt{a}\big)_v = +1\]
where the product is over Hilbert symbols at every place of $\QQ(\sqrt{a})$. Use this to give an alternate proof of Proposition \ref{prop:dih_rec}.
\end{exc}

We now apply this theory to study $8$-class groups. We start with theory first studied by R{\'e}dei and Reichardt in \cite{ReRe33}.

Take $\Delta$ to be the discriminant of an imaginary quadratic fields $K_{\Delta} = \QQ(\sqrt{\Delta})$. From Gauss genus theory, the two torsion of the class group $\text{Cl}\, K_{\Delta}$ is represented by the ideals with squarefree norm $b$ dividing $\Delta$. Each element of the class group will be represented by exactly two $b$ due to the relation
\[(\sqrt{-n}) \sim (1)\]
where $n$ is the squarefree part of $-\Delta$.

The $2$-torsion of the dual group
\[\widehat{\text{Cl}}\, K_{\Delta} = \text{Hom}(\text{Cl}\, K_{\Delta}, \mathbb{C})\]
can also be described. Per class field theory, the set of elements of order two in this group correspond to the set of unramified quadratic extensions of $K_{\Delta}$. These consist precisely of the extensions $K_{\Delta/a, \, a}$ where $a$ is a divisor of $\Delta$ so either $a$ or $\Delta/a$ is $1$ mod $4$. Again, two $a$ correspond to each element of $2$-torsion since $K_{\Delta/a, \, a} = K_{a,\, \Delta/a}$.

We have a perfect pairing
\[\text{Cl} \, K_{\Delta} \times \widehat{\text{Cl}} \, K_{\Delta} \rightarrow \mathbb{C}\]
given by the Artin map. The left and right kernels of the associated map
\[\text{Cl} \, K_{\Delta}[2] \times \widehat{\text{Cl}} \, K_{\Delta}[2] \rightarrow \pm 1\]
are then $2(\text{Cl}\,K_{\Delta} [4])$ and $2(\widehat{\text{Cl}}\,K_{\Delta} [4])$. From this, we get that an ideal of norm $b | \Delta$ is represented by the square of an element in $\text{Cl}\, K_{\Delta}$ if
\[(b, \Delta)_v = +1 \text{ at all rational places } v.\]
Similarly, $K_{\Delta/a, \, a}$ corresponds to a square in the dual class group if
\[(a, -\Delta)_v = +1 \text { at all rational places } v.\]
In this case, if we define $L_{\Delta/a,\,  a}$ as at the beginning of this section, it corresponds to an character of order four with square corresponding to $K_{\Delta/a, \, a}$.

We have a final natural pairing 
\[2(\text{Cl} \, K_{\Delta}[4]) \times 2(\widehat{\text{Cl}} \, K_{\Delta}[4]) \rightarrow \pm 1\]
given by $(x, \phi) \mapsto \psi(x)$, where $\psi$ is any character with $\psi^2 = \phi$. Since $x$ is a square, this map is well defined. Furthermore, its left kernel is $4(\text{Cl} \, K_{\Delta}[8])$, whose rank is the $8$-class rank of this field.

Choose $a$ corresponding to an element of $2(\widehat{\text{Cl}} \, K_{\Delta}[4])$, and choose $b$ corresponding to an element of $2(\text{Cl} \, K_{\Delta}[4])$. Take $\mathbf{b}$ to be an ideal of norm $b$ in $K_{\Delta}$. Then the above pairing can be written in the form
\[\langle a, b \rangle_{\Delta} = \left[\frac{L_{\Delta/a,\, a}/K_{\Delta}}{\mathbf{b}}\right].\]
We will reidentify this symbol as lying in the group $\frac{1}{2}\Z/\Z$, rather than in $\pm 1$.

Now, suppose $\Delta'$ is the discriminant of another imaginary quadratic field, with $a, b | \Delta'$. We also assume that $\Delta\Delta'$ is a square in $\QQ_2$ and that we again have $(a, -\Delta')_v = (b, \Delta')_v = +1$ everywhere. Then we can also consider the pairing $\langle a, b \rangle_{\Delta'}$. Writing $c$ for the squarefree part of $\Delta\Delta'$, we have the crucial equality
\begin{equation}
\label{eq:hurrrg}
\langle a, b \rangle_{\Delta'} = \langle a, b\rangle_{\Delta} + \langle a, b \rangle_{ac},
\end{equation}
To prove this, suppose that $L_{c, a}$ and $L_{\Delta/a,\, a}$ are generated by $\sqrt{x_c + y_c\sqrt{a}}$ of norm $c$ and $\sqrt{x + y\sqrt{a}}$ of norm $\Delta/a$ respectively. Then we note that
\[\sqrt{x' + y'\sqrt{a}} = \sqrt{x_c + y_c\sqrt{a}} \cdot \sqrt{x + y\sqrt{a}}\]
has norm $\Delta'$. The only issue is that $x', y'$ may have a common divisor that divides $\Delta$ and $c$ but not $\Delta'$, so that this specific root would not generate an unramified $L_{\Delta'/a,\, a}/K_{\Delta'/a,\, a}$. Fortunately,  scaling $x_c$ and $y_c$ by this divisor solves this problem. We need to check $2$ separately, but this is easy since $c$ equals $1$ mod $8$, so 
\[\QQ_2\left(\sqrt{a}, \sqrt{\Delta}\right) = \QQ_2\left(\sqrt{a}, \sqrt{\Delta'}\right)\]
and
\[\QQ_2^{unr}\left(\sqrt{a}, \sqrt{\Delta}, \sqrt{x + y\sqrt{a}}\right) = \QQ_2^{unr}\left(\sqrt{a}, \sqrt{\Delta'}, \sqrt{x' + y'\sqrt{a}}\right).\] 
Because we now have an inclusion
\[L_{\Delta'/a, \, a} \subset L_{\Delta/a,\, a}L_{c,\, a},\]
we can take
\[\gamma_{\Delta'/a,\, a} = \gamma_{\Delta/a,\, a} + \gamma_{c,a}\]
and equation \eqref{eq:hurrrg} follows easily.

Now, the conditions on $a$ and $b$ are strong enough to force $(a, b)_v$ to equal $+1$ at all places. Then, via Proposition\ref{prop:dih_rec}, \eqref{eq:hurrrg} becomes
\begin{equation}
\label{eq:gov_2c4p}
\langle a, b \rangle_{\Delta'} = \langle a, b\rangle_{\Delta} + \left[ \frac{L_{a, b}}{\mathbf{c}}\right]
\end{equation}
where $\mathbf{c}$ is an ideal of norm $c$ in $K_{a, b}$.

With this, we can easily construct governing fields by taking $\Delta = dp$ and $\Delta' = dp'$, where $p$ and $p'$ are distinct primes not dividing $2d$. It is tedious and unrewarding to determine how the $8$-class group structure varies over a varying $4$-class structure, so we only do this for $p$ and $p'$ coming from the same quadratic residue class.
\begin{prop}
\label{prop:8c_gov}
\cite{Cors07, Rede39}
Take $d$ to be a negative integer, and let $p_0$ be any prime not dividing $2d$ so $\Delta_0 = dp_0$ is the discriminant of a quadratic imaginary field. Let $p_2, \dots, p_r$ be the primes dividing $d$, and take
\[K = \QQ\left(\sqrt{-1}, \sqrt{2}, \sqrt{p_2}, \dots, \sqrt{p_r}\right).\]

Take $p$ to be a prime not dividing $2d$, and write $\Delta = dp_0$. Suppose $pp_0$ is a square mod $8d$; this is true if and only if $p$ and $p_0$ have the same Artin class in $\text{Gal}(K/\QQ)$. Take $L$ to be the composition of all fields $L_{a, b}$, where $a$, $b$ are any integers dividing $d$ so that
\[(a, -\Delta_0)_v = (b, \Delta_0)_v = +1 \text{ at all places } v.\]
Then the isomorphism class of the group $\text{Cl}\, K_{\Delta}[8]$ is determined by the Artin class of $p$ in $\text{Gal}(L/\QQ)$. 
\end{prop}
\begin{proof}
Note that each element of $\text{Cl}\, K_{\Delta_0}[2]$ has a unique representative as a norm $b$ dividing $d$, and each element of $\widehat{\text{Cl}}\, K_{\Delta_0}[2]$ has a unique representative via an $a$ dividing $d$. Let $V_{\text{Tor}}(d, p_0)$ denote the subspace of $\QmodSq$ generated by the $b$ corresponding to $2(\text{Cl}\, K_{\Delta_0}[4])$, and let $V_{\text{Quo}}(d, p_0)$ denote the subspace generated by the $a$ corresponding to $2(\widehat{\text{Cl}}\, K_{\Delta_0}[4])$. Note that, under the conditions of the proposition,
\[V_{\text{Tor}}(d, p_0) = V_{\text{Tor}}(d, p) \quad \text{ and } \quad
V_{\text{Quo}}(d, p_0) = V_{\text{Quo}}(d, p).\]
Call the mutual dimension of these spaces $m$. The pairing
\[ \langle \quad, \,\,\,\,\,\rangle_{\Delta} : V_{\text{Tor}}(d, p_0) \times V_{\text{Quo}}(d, p_0)  \rightarrow \pm 1\]
has a left kernel isomorphic to $4(\text{Cl}\, K_{\Delta}[8])$ and thus determines the $8$-class rank of $K_{\Delta}$. But \eqref{eq:gov_2c4p} implies that the matrix
\[\langle \quad, \,\,\,\,\,\rangle_{\Delta} - \langle \quad, \,\,\,\,\,\rangle_{\Delta_0}. \]
is determined by the Artin symbol $[p, L/\QQ]$.
This proves the proposition
\end{proof}

This proof suggests something stronger than the proposition it proved. In the case that $L/K$ is an extension of degree $2^{m^2}$, we find that every possible $m \times m$ matrix is represented by a unique Artin class $[p, L/\QQ]$ that restricts to $[p_0, K/\QQ]$ in $\text{Gal}(K/\QQ)$. In this case, we can find the distribution of $8$-class ranks as $p$ varies with Chebotarev's density theorem. With this in mind, we define generic choices of $(d, p_0)$.
\begin{defn}
\label{defn:class_gen}
Take $d$ to be a negative integer, and take $p_0$ to be a prime not dividing $2d$ so $\Delta_0 = dp_0$ is the fundamental discriminant of a quadratic field $K_{dp_0}$. We call $(d, p_0)$ \emph{generic} if $V_{\text{Tor}}(d, p_0)$ and $V_{\text{Quo}}(d, p_0)$ are disjoint vector spaces; that is, if there is no nonsquare $a | d$ so that
\[(a, -\Delta_0)_v = (a, \Delta_0)_v = +1\]
holds at all places $v$.
\end{defn}

We will prove in Lemma \ref{lem:generic} that almost all $(d, p_0)$ are generic, justifying the use of the term. $8$-class ranks behave nicely in generic families, as we see with the following proposition.
\begin{prop}
\label{prop:ineff_1}
Take $\Delta_0 = dp_0$ as in Proposition \ref{prop:8c_gov}, and let $L$ and $K$ be as in that theorem. Write $m$ for the $4$-class rank of $\QQ(\sqrt{\Delta_0})$, and take $0 \le j \le m$. Define
\[P^{\text{Mat}}(j \, | \, m) = \frac{ \left| \{M \in M_m(\mathbb{F}_2) : \,\text{corank}(M) = j\}\right|}{\left| M_m(\mathbb{F}_2)\right|}\]
where $M_m(\mathbb{F}_2)$ is the set of $m\times m$ matrices with coefficients in $\mathbb{F}_2$. Write $Y_{N, \,d, \,p_0}$ for the set of primes $p$ less than $N$ so that $pp_0$ is a square mod $8d$. Then, if $(d, p_0)$ is generic, we have that $L/K$ has degree $2^{m^2}$, and we have
\[ \lim_{N \rightarrow \infty} \frac{\left| \{ p \in Y_{N, \,d, \,p_0}:\,\,  4(\text{Cl}\, K_{dp}[8]) = (\Z/2\Z)^j \}\right|}{\left|  Y_{N, \,d, \,p_0} \right|}= P^{\text{Mat}}(j \, | \, m)  .\]
\end{prop}

\begin{proof}
Take $G = \text{Gal}(L/\QQ)$. Write $a_1, \dots, a_m$ for a basis of $V_{\text{Quo}}(d, p_0)$, and write $b_1, \dots, b_m$ for a basis of of $V_{\text{Quo}}(d, p_0)$. Let $\sigma_i$ be an automorphism of $L$ that fixes all elements of 
\[\left\{\sqrt{a_1}, \dots, \sqrt{a_m}, \sqrt{b_1}, \dots, \sqrt{b_m}\right\}\]
other than $\sqrt{a_i}$, which is sent to $-\sqrt{a_i}$. Define $\tau_j$ to fix everything in this set besides $\sqrt{b_j}$, which is sent to $-\sqrt{b_j}$. The existence of such $\sigma_i$ and $\tau_j$ comes from the definition of a generic $(d, p_0)$.

But then the commutator $[\sigma_i, \tau_j]$ fixes all $L_{a_l, \,b_k}$ other than $L_{a_i, \,b_j}$, and is the unique nontrivial automorphism of $L_{a_i, \,b_j}/K_{a_i, \,b_j}$. From this, we find that $[G, G]$ has order $2^{m^2}$, being generated by the $m^2$ commutators $[\sigma_i, \tau_j]$. $K/\QQ$ is abelian, so $L/K$ then has its maximal possible degree of $2^{m^2}$. Each possible Artin class $[p, L/\QQ]$ corresponds uniquely to a matrix $\langle \quad, \,\,\,\,\,\rangle_{dp}$, and the rest of the proposition follows by the Chebotarev density theorem.
\end{proof}

\section{Governing Fields for $4$-Selmer groups}
\label{sec:4sel}
In this section, we will consider the Cassels-Tate pairing between the $2$-Selmer group of an abelian variety $A/k$ and its dual variety $A^{\vee}$, where $k$ is a number field. Let $A^{(d)}$ denote the quadratic twist of $A$ by the quadratic character $\chi_d$. Notice that $A[2]$ and $A^{(d)}[2]$ are isomorphic over $k$, so we have an identification
\[H^1(G_k, A[2]) \cong H^1(G_k, A^{(d)}[2])\]
where $G_k$ is the absolute Galois group of $k$. In particular, it makes sense to ask if elements $F \in H^1(G_k, A[2])$ and $F' \in H^1(G_k, A^{\vee}[2])$ lie in the $2$-Selmer groups of $A^{(d)}$ and $(A^{(d)})^{\vee}$ respectively. If so, we denote the value of the Cassels-Tate pairing of $F$ and $F'$ over $A^{(d)}$, $(A^{(d)})^{\vee}$ by
\[\langle F, F' \rangle_{A^{(d)}}.\]
This pairing will take values in $\frac{1}{2}\Z/\Z$. It will \emph{usually} depend on the specific value of $d$. However, the pairing will be alternating with respect to a polarization coming from a $k$-rational divisor \cite{Poon99}, and we do not expect this property to be affected by a quadratic twist. We account for this situation with the following lemma.
\begin{lem}
\label{lem:alt_anti}
Take $A/k$ to be an abelian variety with full $2$-torsion over $k$. Take $(\,\,\,\, , \,\,\,)_{\text{Weil}}$ to be the Weil pairing 
\[(\,\,\,\, , \,\,\,)_{\text{Weil}}: A[2] \times A^{\vee}[2] \rightarrow \pm 1,\]
and use this to define a cup product 
\[\cup: H^1(G_k, A[2]) \times H^1(G_k, A^{\vee}[2]) \rightarrow H^2(G_k, \pm 1).\]
Take $F \in H^1(G_k, A[2])$ and $F' \in H^1(G_k, A^{\vee}[2])$.  Call $F \cup F'$ antisymmetric if
\[(F(\sigma), F'(\tau))_{\text{Weil}} = (F(\tau), F'(\sigma))_{\text{Weil}}\]
for all $\sigma, \tau \in G_k$. If $F \cup F'$ is antisymmetric, call it alternating if
\[(F(\sigma), F'(\sigma)) = +1\]
for all $\sigma$. Write $K/k$ for the minimal field over which $F$ and $F'$ are trivial cocycles.

Then the cocycle $F \cup F'$ is zero in $H^2(\text{Gal}(K/k), \pm 1)$ if and only if it is alternating. Taking $k^{ab}$ to be the maximal abelian extension of $k$, we also find that $F \cup F'$ can only be zero in $H^2(\text{Gal}(k^{ab}/k), \pm 1)$ if it is antisymmetric.
\end{lem}
\begin{proof}
Note that $K/k$ is a $(\Z/2\Z)^r$ extension for some $r$. Choose some basis $g_1, \dots, g_r$ for this group, and take $A$ to be the $r \times r$ matrix with coefficients in $\mathbb{F}_2$ given by
\[A_{ij} = \frac{1}{2}\big(1 - (F(g_i), F'(g_j))_{\text{Weil}} \big).\]
Take $K = k(\sqrt{c_1}, \dots, \sqrt{c_r})$ with $g_i$ fixing all $\sqrt{c_j}$ besides $\sqrt{c_i}$. Take $\kappa_{ij}: G_k \rightarrow \pm 1$ to be the $1$ cochain mapping $\sigma$ to $-1$ if and only if $\sigma(\sqrt{c_i}) = -\sqrt{c_i}$ and $\sigma(\sqrt{c_j}) = -\sqrt{c_j}$. Then, if $F \cup F'$ is alternating, $A$ is also alternating, and the $1$ cochain
\[\sum_{i > j} A_{ij} \kappa_{ij}\]
has coboundary equal to $F \cup F'$. On the other hand, suppose any diagonal entry of $A$ was nonzero; say $A_{11}$ was nonzero. Then the restriction of $F \cup F'$ to the group $\{1, g_1\}$ is the nontrivial element of $H^2(\Z/2\Z, \Z/2\Z)$, and $F \cup F'$ is thus nontrivial as a cocycle over $\text{Gal}(K/k)$. Finally, if $F \cup F'$ is not antisymmetric, we see that no cochain $\Gamma$ defined on an abelian extension of $k$ can have coboundary $F \cup F'$, as this would imply
\[(F(\sigma), F'(\tau))_{\text{Weil}} \cdot (F(\tau), F'(\sigma))_{\text{Weil}} = \Gamma(\sigma \tau)\Gamma(\tau\sigma) = +1.\]
everywhere.
\end{proof}

With this lemma out of the way, we can construct governing fields for $4$-Selmer groups.
\begin{thm}
\label{thm:4sel_gov}
Take $A$ to be an abelian variety above a number field $k$. Assume $A$ has full $2$-torsion over $k$. Take
\[F \in H^1(G_k, A[2]), \quad F' \in H^1(G_k, A^{\vee}[2]).\]
Take $K/k$ to be the minimal field over which $F$ and $F'$ are trivial.

Next, take $S$ to be any set of places of $k$ with the following properties:
\begin{itemize}
\item $S$ contains all places of bad reduction of $A$.
\item $S$ contains all archimedean primes and all primes dividing $2$.
\item $S$ contains enough primes to generate the quotient $\text{Cl}\, k /2\text{Cl}\, k$ of the class group of $k$.
\end{itemize}

Now take $\mathscr{D}$ to be the set of pairs $(d_1, d_2)$ of elements in $k^{\times}$ such that
\begin{itemize}
\item $d_1/d_2$ is a square at all places of $S$ and
\item $F$ and $F'$ are $2$-Selmer elements of $A^{(d_i)}$, $\left(A^{\vee}\right)^{(d_i)}$ respectively for $i = 1, 2$.
\end{itemize}

Then, if $F \cup F'$ is alternating in the sense detailed above, we have
\[\langle  F, F' \rangle_{A^{(d_1)}} = \langle  F, F' \rangle_{A^{(d_2)}} \]
for all $(d_1, d_2) \in \mathscr{D}$. 

If $F \cup F'$ is not alternating, there is a quadratic extension $L$ of $K$ that is ramified only at primes in $S$ such that
\[\langle F, F' \rangle_{A^{(d_1)}} = \langle F, F' \rangle_{A^{(d_2)}} + \left[ \frac{L/K}{\mathbf{d}} \right]\]
for all $(d_1, d_2) \in \mathscr{D}$, where the Galois group $\text{Gal}(L/K)$ is identified here with $\frac{1}{2}\Z/\Z$. In this equation, $\mathbf{d}$ is any ideal of $K$ coprime to the conductor of $L/K$ that has norm $(d_1/d_2)\mathbf{b}^2$ in $k$ for some ideal $\mathbf{b}$ of $k$; such a $\mathbf{d}$ exists for any $(d_1, d_2)$ in $\mathscr{D}$.
\end{thm}

\begin{proof}
We use the Weil pairing definition of the Cassels-Tate pairing as given in \cite[I.6.9]{Milne86} and \cite[12.2]{Poon99}. Choose 
\[\beta_1 \in \text{Map}(G_k, \,A^{(d_1)}[4])\]
so that $2\beta_1 = F$.  Then the coboundary $d\beta_1$ lies in $H^2(G_k, A^{(d_1)}[2])$. Using the cup product
\[\cup: H^2(G_k, A[2]) \times H^1(G_k, A^{\vee}[2]) \rightarrow H^3(G_k, \pm 1)\]
coming from the Weil pairing, we can define a $3$-cocycle $d\beta_1 \cup F'$. Taking $k_s$ to be the separable closure of $k$, we know 
\[H^3(G_k, k_s^{\times}) = 0,\]
so $d\beta_1 \cup F' = d\epsilon_1$ for some $2$-cochain $\epsilon_1$. Finally, for every place $v$, there is some
\[\beta_{v1} \in \text{ker}\big(H^1\left(G_{k_v}, A^{(d_1)}[4]\right) \rightarrow H^1\left(G_{k_v}, A^{(d_1)}\right)\big)  \]
with $2\beta_{v1} = F_v$. Then $(\beta_{1v} - \beta_{v1}) \cup F'_v  - \epsilon_{1v}$ is a $2$-cocycle, and we define
\[\langle F, F' \rangle_{A^{(d_1)}} = \sum_v \text{inv}_v\big((\beta_{1v} - \beta_{v1}) \cup F'_v  - \epsilon_{1v}\big).\]

Write $d = d_1d_2$, and take $\phi: A^{(d_1)} \rightarrow A^{(d_2)}$ to be the $k(\sqrt{d})$ isomorphism coming from the twist. For the Cassels-Tate pairing over $A^{(d_2)}$, we need
\[\beta_2 \in \text{Map}(G_k, \,A^{(d_1)}[4])\]
with $2\beta_2 = F$. Such a $\beta_2$ is provided by $\phi \beta_1$, so we set $\beta_2 = \phi \beta_1$. Then we need to find an $\epsilon_2$ so
\[d(\epsilon_1 - \epsilon_2) = (\phi d\beta_1 - d\phi \beta_1) \cup F'.\]
But
\begin{align*}
& (\phi d\beta_1 - d\phi \beta_1)(\sigma, \tau) \\
& = \phi\big(\sigma \beta_1(\tau)  \,\,\,\,\,\, -\beta_1(\sigma\tau) \,\, + \beta_1(\sigma)\big) \\
& \,\quad - \big(\sigma \phi \beta_1(\tau) - \phi\beta_1(\sigma\tau) + \phi\beta_1(\sigma) \big) \\
& = (\phi \sigma - \sigma \phi) \beta_1(\tau) \\
&  = \begin{cases} \sigma F(\tau) &\text{ if } \sigma(\sqrt{d}) = -\sqrt{d} \\ 0 & \text{ otherwise.} \end{cases}
\end{align*}
Since the Galois action of $G_k$ on $A[2]$ is trivial, we can thus write
\begin{equation}
\label{eq:eps_dif}
\phi d\beta_1 - d\phi \beta_1 = \chi_d \cup F
\end{equation}
where $\chi_d$ is the quadratic character corresponding to $k(\sqrt{d})/k$ and where the cup product comes from the unique nontrivial bilinear map $\Z/2\Z \times \Z/2\Z \rightarrow \Z/2\Z$.

Per \cite[Lemma I.6.15]{Milne86}, $\text{inv}_v(F \cup F')$ is everywhere zero. Since the map
\[H^2(G_k, k_s^{\times}) \rightarrow \bigoplus_v H^2(G_{k_v}, \overline{k}_v^{\times})\]
is injective, $F \cup F'$ is the coboundary of some $1$-cochain 
\[\Gamma: G_k \rightarrow k_s^{\times}.\] 
Then we can take
\[\epsilon_2  =\epsilon_1 \,-\,\chi_d \cup \Gamma .\]
We have
\[\text{inv}_v\big( (\beta_{1v} - \beta_{v1}) \cup F'_v  - \epsilon_{1v}\big) = \text{inv}_v\big( (\phi\beta_{1v} - \phi\beta_{v1}) \cup F'_v  - \epsilon_{1v}\big)\]
and can thus write
\begin{equation}
\label{eq:ct_dif}
\langle F, F' \rangle_{A^{(d_1)}} - \langle F, F' \rangle_{A^{(d_2)}} = \sum_v \text{inv}_v\big((\beta_{v2} - \phi\beta_{v1}) \cup F'_v + \chi_d \cup \Gamma\big).
\end{equation}

We want to find as simple a $\Gamma$ as possible. From the exact sequence 
\[1 \xrightarrow{\quad\,} \pm 1 \xrightarrow{\quad\,} k_s^{\times} \xrightarrow{\,\,\,2\,\,\,} k_s^{\times} \xrightarrow{\quad\,} 1,\]
we get that
\[H^1(G_k, k_s^{\times}) \rightarrow H^2(G_k, \pm 1) \rightarrow H^2(G_k, k_{s}^{\times})\]
is exact. From Hilbert 90, we get that the second map in this sequence must be injective. Then we can assume that $\Gamma$ takes values in $\pm 1$.

Suppose $\Gamma$ is defined on an extension $L_0/K$ with $L_0$ Galois over $k$.  Take $G$ to be $\text{Gal}(L_0/k)$ and $N$ to be the subgroup of $G$ fixing $K$. For $n \in N$ and $g \in G$, we have $\Gamma(n) + \Gamma(g) = \Gamma(ng)$, as $d\Gamma$ is defined on $G/N$. Then $\Gamma$ is a character on $N$; take $N_1$ to be its kernel in $N$. Take $n \in N_1$. Then
\[\Gamma(gng^{-1}) \,=\, (F \cup F')(g, ng^{-1}) + \Gamma(g) + \Gamma(ng^{-1}) \]
\[ \quad\,\,\,= (F \cup F')(g, g^{-1}) + \Gamma(g) + \Gamma(g^{-1})\, =\, \Gamma(gg^{-1}),\]
which is the identity. Because of this, $N_1$ must be normal in $G$. Furthermore, $\Gamma$ is defined on $G/N_1$. We can then assume that $\Gamma$ is either defined on $K$ or on a quadratic extension $L_0$ of $K$ that is Galois over $k$.

In the latter case, write $L_0 = K(\sqrt{\alpha})$. Write $\mathfrak{d}$ for the product of primes in $K$ that ramify in $L_0/K$ but that are not above a ramified prime of $K/k$. We see that $\mathfrak{d}$ is preserved by the action of $\text{Gal}(K/k)$. Since no prime dividing $\mathfrak{d}$ ramifies in $K/k$, $\mathfrak{d}$ is an ideal of $k$. Choose $\alpha_0 \in k^{\times}$ and $\mathbf{b}$ a nonzero ideal of $k$ so $(\alpha_0) \mathfrak{d}\mathbf{b}^2$ is an ideal generated by primes in $S$, and replace $\Gamma$ with $\Gamma + \chi_{\alpha_0}$. This change does not alter the coboundary of $\Gamma$, but it changes the cochain's field of definition to $L = K(\sqrt{\alpha\alpha_0})/k$. Then $\Gamma$ is now defined on a field $L/K$ that is potentially ramified only at the places in $S$.

Suppose that $v$ is a place outside $S$ where $\QQ(\sqrt{d_i})$ ramifies. We claim that $F$ and $F'$ are trivial over $k_v$. For take $k_4/k$ to be the minimal field such that $G_{k_4}$ acts trivially on $A[4]$ and $A^{\vee}[4]$. Then $k_4/k$ is unramified at $v$ since $v$ is outside $2$ and the primes of bad reduction of $A$ (see \cite[Theorem 1]{SeTa68}, for example). The image of $A[2]$ in the connecting morphism
\[H^0(G_{k_v}, A) \rightarrow  H^1(G_{k_v}, A[2])\]
thus consists of unramified cocycles. Take $I_{k_v}$ to be the inertia group of $G_{k_v}$. Then we have that, via the composition
\[H^0(G_{k_v}, A^{(d_i)}) \rightarrow  H^1(G_{k_v}, A^{(d_i)}[2]) \rightarrow  H^1(I_{k_v}, A^{(d_i)}[2]),\]
$A^{(d_i)}[2]$ maps bijectively to  $H^1(I_{k_v}, A^{(d_i)}[2])$. Since the Weil pairing is nondegenerate, and since
\[\text{inv}_v(T \cup F') = 0\]
for $T \in H^1(G_{k_v}, A^{(d_i)}[2])$ coming from $H^0(G_{k_v}, A^{(d_i)})$ \cite[Lemma I.6.15]{Milne86}, and since $K$ is unramified at $v$, we have that $F'$ is trivial over $k_v$. The same argument implies that $F$ is trivial over $k_v$. In particular, $K/k$ splits completely at all primes with odd multiplicity in either $d_i$ outside $S$. With our conditions on $\mathscr{D}$, we thus can prove the existence of $\mathbf{d}$ in $K$ with norm $(d_1/d_2)\mathbf{b}^2$ in $k$.

We return to \eqref{eq:ct_dif}. The summand is the invariant of an unramified cocycle at all places outside $S$ that divide neither $d_i$. It is hence zero at such places.

At places in $S$, we see that $\phi$ is an isomorphism. We can thus take $\beta_{v2} = \phi \beta_{v1}$ at such a place. Furthermore, $\chi_d$ is trivial at such a place since $d_1/d_2$ is a local square. These places thus do not contribute.

At places $v$ outside $S$ that divide $d$ with even multiplicity,  we see that $\chi_d$ and $\Gamma$ are both locally unramified. If the multiplicity of $v$ is even in both $d_1$ and $d_2$, we have that $A^{(d_1)}$ and $A^{(d_2)}$ have good reduction at $v$, so $\beta_{v2} - \phi \beta_{v1}$ and $F'_v$ are unramified. If the multiplicity is odd for both, we have that $F'_v$ is trivial. Either way, these summands do not contribute.

We are left with finite places $v$ that divide $d$ with odd multiplicity. At such places, $F'_v$ is trivial. Then \eqref{eq:ct_dif} has become
\begin{equation}
\label{eq:ct_dif_2}
\langle F, F' \rangle_{A^{(d_1)}} - \langle F, F' \rangle_{A^{(d_2)}} = \sum_{v \text{ with } v(d) \text{ odd}} \text{inv}_v(\chi_d \cup \Gamma).
\end{equation}

We now consider the impact of the symmetry of $F \cup F'$ on this equation using Lemma \ref{lem:alt_anti}. In the alternating case, $\Gamma$ is defined over $K$ and is hence zero at all places being summed over in the above equation. Then, if $F \cup F'$ is alternating,
\[\langle F, F' \rangle_{A^{(d_1)}}  = \langle F, F' \rangle_{A^{(d_2)}}. \]
Otherwise, $\Gamma$ is defined over a quadratic extension $L/K$, and the above equation counts the number of places of odd valuation in $d$ where $L/K$ is nonsplit. This is precisely what is calculated with the Artin symbol
\[ \left[ \frac{L/K}{\mathbf{d}} \right],\]
and we have the theorem.
\end{proof}

\subsection{Theorem \ref{thm:4sel_gov} for elliptic curves over $\QQ$}
In this subsection, we specialize Theorem \ref{thm:4sel_gov} to the case of an elliptic curve $E/\QQ$ with full $2$-torsion. In this case, the analogy between $4$-Selmer group and $8$-class groups becomes particularly simple. We will be able to define the notion of a generic twist $E^{(dp_0)}$ as we had defined generic discriminants $\Delta_0 = dp_0$.

As a starting point, we consider an alternative approach to Theorem \ref{thm:4sel_gov} for this special case. In the proof, we were able to find $\Gamma$ defined over a quadratic extension $L$ of $K$ using Hilbert 90 and a basic group theory trick. This avoids some computational hurdles, especially for higher dimensional abelian varieties, but it also is an abstract way to construct $L/K$, and it partially obstructs the analogy with the governing field of the $8$-class group.

With this in mind, we give an alternative approach to finding $\Gamma$. Take $E/\QQ$ to be an elliptic curve with full $2$-torsion. Take $t_1$ and $t_2$ to be generators for the $2$-torsion. Then there are squarefree integers $a_1$, $a_2$, $a_1'$, and $a_2'$ so that we can write the Selmer elements $F$ and $F'$ in the form
\[F(\sigma) = \chi_{a_1}(\sigma)t_1 + \chi_{a_2}(\sigma)t_2 \quad\text{ and } \quad F'(\sigma) = \chi_{a'_1}(\sigma)t_1 + \chi_{a'_2}(\sigma)t_2\]
Write $a_3$ and $a'_3$ for the squarefree part of $a_1a_2$ and $a'_1a'_2$ respectively. The tuples $(a_1, a_2, a_3)$ and $(a'_1, a'_2, a'_3)$ are the objects of choice for classical two descent (see \cite{Swin08}, for example). In this case, we see the equation
\[\text{inv}_v(F \cup F') = +1\]
is equivalent to the relation
\[(a_1, a'_1)_v \cdot (a_2, a'_2)_v \cdot (a_3, a'_3)_ = +1\]
of Hilbert symbols at any place $v$. From this, we can find $b \in \QQ^{\times}$ so that
\[(a_1, ba'_1)_v = (a_2, ba'_2)_v = (a_3, ba'_3)_v = +1\]
at all places. With Hasse-Minkowski, we can then find $x_i, y_i, z_i \in \QQ^{\times}$ so
\[x_i^2 - a_iy_i^2 = ba'_i\,z_i^2.\]
Assume that $K_{F, F'} = \QQ(\sqrt{a_1}, \sqrt{a_2}, \sqrt{a'_1}, \sqrt{a'_2})$ is of degree 16 above $\QQ$. Then we find that
\begin{equation}
\label{eq:deg32}
L_{F, F'} = K_{F, F'}\left(\sqrt{(x_1 + y_1\sqrt{a_1})(x_2 + y_2\sqrt{a_2})(x_3 + y_3\sqrt{a_3})}\right)
\end{equation}
is a degree 32 Galois extension of $\QQ$. We can scale $x_1$, $y_1$, and $z_1$ by a common factor to avoid ramification at places unramified in $K/\QQ$. Then there is a $\Gamma$  defined over $L_{F, F'}$ that has coboundary $F \cup F'$, and $L_{F, F'}/\QQ$ serves as a governing field as in Theorem \ref{thm:4sel_gov}.

Next, given an odd prime $p_0$ not dividing the conductor of $E$, we will define a vector space
\[W_{\text{SD}}(E, p_0) \subseteq \QmodSq \,\times\, \QmodSq\]
that is isomorphic to the quotient of $\SelT(E^{(p_0)})$ by the image of $E[2]$ in the map
\[H^0(G_{\QQ}, E) \rightarrow H^1(G_{\QQ}, E[2]).\]
For $F \in \SelT(E^{(p_0)})$, there is a unique element $F_{\text{tor}}$ coming from $E[2]$ so that $F + F_{\text{tor}}$ is unramified as a cocycle at $p_0$. Then $F + F_{\text{tor}}$ corresponds to a tuple $(a_1, a_2, a_3)$ with none of the $a_i$ divisible by $p_0$. With this setup, we say that $F$ maps to $(a_1, a_2)$ in $W_{\text{SD}}(E, p_0)$, and we take $W_{\text{SD}}(E, p_0)$ to be the minimal space containing all the tuples corresponding to $2$-Selmer elements

Note that, if $pp_0$ is a square mod $8C$, where $C$ is the conductor of $E$, then
\[W_{\text{SD}}(E, p_0) = W_{\text{SD}}(E, p).\]
Thus, with this space, we quickly get the analogue of Proposition \ref{prop:8c_gov}.
\begin{prop}
\label{4s_gov}
Take $E/\QQ$ to be an elliptic curve with full $2$-torsion, and take $C$ to be the conductor of $E$. Take $K/\QQ$ to be the minimal field containing $\sqrt{d}$ for all divisors $d$ of $2C$. Take $p_0$ to be a prime not dividing $2C$, and suppose $p$ is any other prime so $pp_0$ is a square mod $8C$. In this case, we get $[p, K/\QQ] = [p_0, K/\QQ]$.

For distinct tuples $(a_1, a_2)$ and $(a'_1, a'_2)$ in $W_{\text{SD}}(E, p_0)$, find $L_{F, F'}$ as in \eqref{eq:deg32}, and take $L$ to be the minimal extension of $K$ containing all these extensions. 

Then the Cassels-Tate pairing on $W_{\text{SD}}(E, p)$, and thus the rank of $\text{Sel}^{(4)}(E^{(p)})$, is determined uniquely by the symbol $[p, L/\QQ]$.
\end{prop}

Supposing $W_{\text{SD}}(E, p_0)$ has rank $m$, there are $2^{\frac{1}{2}(m^2 - m)}$ possible alternating bilinear forms on this space. If this equals the degree of $L/K$, we find that there is a bijective correspondence between $\text{Gal}(L/K)$ and the set of alternating matrices, and we can use Chebotarev's density theorem in analogy to Proposition \ref{prop:ineff_1}. This merits the definition of generic twists.
\begin{defn}
\label{defn:ell_gen}
Take $E/\QQ$ to be an elliptic curve with full $2$-torsion, and take $p_0$ to be an odd prime not dividing the conductor of $E$. Then we call $(E, p_0)$ a \emph{generic twist} if the projection maps
\[\pi_i: W_{\text{SD}}(E, p_0) \rightarrow \QmodSq\]
are injective and have disjoint images for $i = 1, 2$, where $\pi_i$ maps to the $i^{th}$ factor of $\QmodSq \,\times\, \QmodSq$.
\end{defn}
For elliptic curves $E$ that obey certain technical hypotheses, we will find in Lemma \ref{lem:generic} that almost all $(d, p_0)$ will make $(E^{(d)}, p_0)$ a generic twist, so this definition is justified. In analogy with Proposition \ref{prop:ineff_1}, we have the following.

\begin{prop}
\label{prop:ineff_2}
Take $E$, $p_0$, $L$, and $K$ as in the previous proposition. Write $m + 2$ for the $2$-Selmer rank of $E^{(p_0)}$, and for $0 \le j \le m$, define
\[P^{\text{Alt}}(j \, | \, m) = \frac{ \left| \{M \in M^{\text{alt}}_m(\mathbb{F}_2)  : \,\text{corank}(M) = j\}\right|}{\left| M^{\text{alt}}_m(\mathbb{F}_2)\right|}\]
where $M_m^{\text{alt}}(\mathbb{F}_2)$ is the set of alternating $m\times m$ matrices with coefficients in $\mathbb{F}_2$. Write $Y_{N, \,C,\, p_0}$ for the set of primes $p$ less than $N$ so that $pp_0$ is a square mod $8C$. Then, if $(E, p_0)$ is a generic twist, $L/K$ has degree $2^{\frac{1}{2}(m^2 - m)}$ and
\[ \lim_{N \rightarrow \infty} \frac{\left| \{ p \in Y_{N, \,C, \,p_0}:\,\,  2(\text{Sel}^{(4)}(E^{(p)})) = (\Z/2\Z)^j \}\right|}{\left| Y_{N,\,C,\,p_0}\right|}= P^{\text{Alt}}(j \, | \, m)  .\]
\end{prop}
\begin{proof}
Take $G = \text{Gal}(L/\QQ)$. Let 
\[F_1 = (a_{11}, a_{21}) ,\, \dots\,, F_m = (a_{1m}, a_{2m})\]
be a basis for $W_{\text{SD}}(E, p_0)$. Due to the genericity hypothesis, there is a $\sigma_{ij} \in \text{Gal}(L/\QQ)$ so
\[\sigma_{ij}(\sqrt{a_{ij}}) = -\sqrt{a_{ij}}\]
while
\[\sigma_{ij}(\sqrt{a_{kl}}) = \sqrt{a_{kl}}\]
for any $(k, l) \ne (i, j)$.

Now assume that $0 \le i, j \le m$ with $i \ne j$, and take $\Gamma_{ij}$ to be the cocycle on $L_{F_i, F_j}$ with coboundary $F_i \cup F_j$.
Then
\[F_i(\sigma_{1i}) \cup F_j(\sigma_{2j}) \ne F_j(\sigma_{1i}) \cup F_i(\sigma_{2j})\]
so $\Gamma_{ij}([\sigma_{1i}, \sigma_{2j}]) = -1$. But, if $\{l, k\} \ne \{i, j\}$, we find that
\[\Gamma_{lk}([\sigma_{1i}, \sigma_{2j}]) = +1.\]
In particular, the commutator $[G, G]$ has order at least $2^{\frac{1}{2}(m^2 - m)}$, being generated by $[\sigma_{1i}, \sigma_{2j}]$ for all pairs $1 \le i < j \le m$. But this is the maximal possible degree of the extension $L/K$, as $\text{Gal}(L/K)$ injects into the space of alternating pairings on $W_{\text{SD}}(E, p_0)$. This gives the proposition.
\end{proof}

\section{Statistical Implications}
\label{sec:stats}

With Proposition \ref{prop:ineff_1} and Proposition \ref{prop:ineff_2}, we have found how $8$-class ranks for quadratic fields $\QQ(\sqrt{dp})$ and $4$-Selmer ranks for twists $E^{(p)}$ are distributed for $p$ coming from a set $Y_{N,\, d, \,p_0}$ of primes. However, we are more interested in the distribution of these groups in larger families. On one side, the Cohen-Lenstra heuristics give predictions for the distribution of $8$-class groups over all imaginary quadratic fields. Though Delaunay's heuristics were originally framed for all elliptic curves over $\QQ$ ordered by conductor \cite{Del01}, the more natural analogue to the set of all imaginary quadratic fields is the set of all quadratic twists of a given elliptic curve. So these are the families we are interested in understanding.

The easiest way to do this is to split these large families into families covered by Proposition \ref{prop:ineff_1} and Proposition \ref{prop:ineff_2}. For this to work, we would need to make the error terms of these propositions explicit. Unfortunately, the effective form of Chebotarev's density theorem is far too weak for this application without any extra conditions. Because of this, we need some form of the grand Riemann hypothesis.
\begin{ass}
\label{ass:metab}
Let $K/\QQ$ be any finite Galois extension with $\text{Gal}(K/\QQ)$ a metabelian $2$-group. Then the grand Riemann hypothesis holds for the Dedekind zeta function $\zeta_K(s)$ associated to this extension. In other words, every zero of $\zeta_K(s)$ with $0 < \text{Re}(s) < 1$ lies on the line $\text{Re}(s) = \frac{1}{2}$.
\end{ass}
Since the Artin conjecture is known for monomial extensions, this assumption implies the grand Riemann hypothesis for Artin  $L$-functions associated with representations of metabelian $2$-groups.

\begin{thm}
\label{thm:8class_stats}
Suppose that Assumption \ref{ass:metab} is true. Let $X^-_N$ be the set of negative squarefree integers $d$ with $|d| < N$. Recall the notation $K_d = \QQ(\sqrt{d})$ for $d \in X^-_N$. Supposing $m \ge j \ge 0$, define $P^{\text{Class}}(j\, | \,m)$ to be
\[ \lim_{N \rightarrow \infty} \frac{\left|\{d \in X^-_N : \,\,2\big(\text{Cl}\, K_d[8]\big) \,\cong \,(\Z/2\Z)^{m-j} \oplus (\Z/4\Z)^j \} \right|}{\left|\{d \in X^-_N : \,\, 2\big(\text{Cl} \,K_d[4]\big) \cong (\Z/2\Z)^m \}\right|}.\]

Then, for al $m \ge j \ge 0$, we have
\[P^{\text{Class}}(j\, | \,m) = P^{\text{Mat}}(j \, | \, m),\] 
with $P^{\text{Mat}}(j \, | \, m)$ as defined in Proposition \ref{prop:ineff_1}.
\end{thm}

\begin{thm}
\label{thm:4Sel_stats}
Suppose that Assumption \ref{ass:metab} is true. Let $E/\QQ$ be an elliptic curve with full $2$-torsion. Assume that $E$ has no cyclic subgroup of order four defined over $\QQ$.  Let $X_N$ be the set of nonzero squarefree integers $d$ with $|d| < N$. Supposing $m \ge j \ge 0$, define $P^{\text{Selm}}_E(j\, | \,m)$ to be
\[ \lim_{N \rightarrow \infty} \frac{\big|\{d \in X_N  : \,\,\text{Sel}^{(4)}\big(E^{(d)}\big) \,\cong \,(\Z/2\Z)^{m-j + 2} \oplus (\Z/4\Z)^j \} \big|}{\big|\{d \in X_N  : \,\, \text{Sel}^{(2)}\big(E^{(d)}\big) \cong (\Z/2\Z)^{m+2} \}\big|}.\]

Then, for all $m \ge j \ge 0$, we have 
\[P^{\text{Selm}}_E(j\, | \,m) = P^{\text{Alt}}(j \, | \, m),\]
with $P^{\text{Alt}}(j \, | \, m)$ as defined in Proposition \ref{prop:ineff_2}.
\end{thm}

Before turning to the proof of these theorems, we note that they are consistent with the Cohen-Lenstra and BKLPR heuristics. From \cite{FrMa89} and Gerth's work, we expect that the distribution of the groups $2(\text{Cl}\, K_d[2^{\infty}])$ is approached by the distribution of cokernels of $r \times r$ matrices $M$ with coefficients in $\Z_2$ as $r$ heads to infinity. Taking $M$ to be such a matrix, we can write its cokernel as the set of row vectors $v^{\top} \in (\QQ_2/\Z_2)^r$ so that
\[v^{\top}M = 0.\]
The cokernel of its transpose consists of the vectors  $v \in (\QQ_2/\Z_2)^r$ so
\[Mv = 0.\]
Suppose $M$ had cokernel with $2$-rank $m$. Taking $e_1, \dots, e_r$ for a basis of $\QQ_2^r$, we can choose matrices $A, B \in \text{GL}_r(\Z_p)$ so $AMB$ has transpose cokernel and cokernel spanned by the set $\{\frac{1}{2}e_1, \dots, \frac{1}{2}e_m\}$ and their transposes. Define a bilinear pairing with values in $\frac{1}{2}\Z/\Z$ by
\[\langle \frac{1}{2}e_i, \frac{1}{2}e_j\rangle =  \frac{1}{4}e_i^{\top} AMB e_j\]
Different forms of the matrix in $M_m(\mathbb{F}_2)$ defining this pairing correspond to subsets of $M_r(\Z_2)$ with equal Haar measure, and the rank of the kernel of this pairing is the $4$-rank of the cokernel of $M$. So a cokernel with $2$-rank $m$ has probability $P^{\text{Mat}}(j \, | \, m)$ of having $4$-rank $j$, and this is consistent with Theorem \ref{thm:8class_stats}.

The proof that Theorem \ref{thm:4Sel_stats} is consistent with the BKLPR heuristic is largely the same. We just need to work over the family of alternating matrices instead.

\subsection{Proof of the main theorems}
Our initial setup comes largely from \cite{Kane13}. For $D$ a positive even integer, take $S_{N, r, D}$ to be the set of tuples of distinct primes $(p_1, \dots, p_r)$ with $\prod_i p_i$ coprime to $D$ and less than $N$. 

Given a property $P$ of tuples of primes and a positive integer $N$, we define the \emph{limit density} of $P$ in the $S_{N, r, D}$ as
\[\lim_{N\rightarrow \infty} \frac{\sum_r \left| \{ x\in S_{N, r, D}: \,\, P(x, N)\} \right|}{\sum_r | S_{N, r, D}|}.\]
where, here and throughout this section, the sums over $r$ are taken over the range
\[\ln\ln N - (\ln\ln N)^{3/4} <\,\, r \,\, < \ln\ln N + (\ln\ln N)^{3/4}.\]
From \cite[Theorem A]{HaRa17}, we know that, as $N$ heads to infinity, almost all squarefree numbers $n$ less than $N$ have their number of prime factors in this range. Thus, if $P$ does not depend on the order of the $p_i$ or on $N$, the limit density of $P$ is equal to the natural density of positive squarefree numbers coprime to $D$ that satisfy $P$.

The following result is clear from the work of Fouvry-Kl{\"u}ners \cite{Fouv07} and Kane \cite{Kane13}.

\begin{lem}
\label{lem:two_four}
For any $m \ge 0$, the natural density of negative quadratic discriminants $\Delta$ so that $K_\Delta$ has $4$-class rank exactly $m$ exists and is positive

Similarly, take $E/\QQ$ to be an elliptic curve with full rational two torsion. Assume that $E$ has no cyclic subgroup of order four defined over $\QQ$. Then, for any $m \ge 0$, the natural density of squarefree integers $n$ so that $E^{(n)}$ has $2$-Selmer rank exactly $m+2$ exists and is positive.
\end{lem}

Because of this lemma, we can exclude discriminants and twists identified by a property with limit density zero. This is where the following two lemmas come into play. For $x = (p_1, \dots, p_r)$, we write
\[d_{2:r} = \prod_{2 \le i \le r} p_i\]

\begin{lem}
\label{lem:p1_tiny}
For $(p_1, \dots, p_r) \in S_{N, r, D}$, let $P$ be the property that 
\[d_{2:r} > Ne^{-e^{\sqrt{\ln \ln N}}}.\]
Then the limit density of $P$ over the $S_{N, r, D}$ is zero.
\end{lem}

Recall that we defined the notion of a generic discriminant and generic twist in Definitions \ref{defn:class_gen} and \ref{defn:ell_gen}.

\begin{lem}
\label{lem:generic}
Take $d_0$ to be $-1$, $-4$, or $-8$, and take $D = 2$.  Then the limit density of $(p_1, \dots, p_r)$ in $S_{N, r, D}$ with $(d_0d_{2:r}, \,p_1)$ a non-generic discriminant is zero.
 
Similarly, take $E/\QQ$ to be an elliptic curve with full rational two torsion with conductor dividing $D$. Assume that $E$ has no cyclic subgroup of order four defined over $\QQ$. Then the limit density of $x \in S_{N, r, D}$ with $(E^{(d_{2:r})}, \,p_1)$ a non-generic twist is zero.
\end{lem}

We will prove these lemmas after proving that they imply the main theorems.

\begin{proof}[Proof of Theorem \ref{thm:8class_stats}]
Take $D = 2$. Take $d_0 \in \{-1, -4, -8\}$.

Suppose $(p_1, \dots, p_r)$ lies in $S_{N, r, D}$. Write $\Delta = d_0 \cdot \prod_{i = 1}^r p_i$ and $d_{2:r} = \prod_{i=2}^r p_i$. We say that the tuple satisfies the property $P_m$ if $(d_0d_{2:r}, p_1)$ is a generic discriminant, if
\[2\big(\text{Cl} \,K_{\Delta}[4]\big) \cong (\Z/2\Z)^m,\]
and if
\[d_{2:r} \le Ne^{-e^{\sqrt{\ln\ln n}}}.\]
For $0 \le j \le m$, we say that the tuple satisfies $P_{mj}$ if it satisfies $P_m$ and if
 \[2\big(\text{Cl}\, K_{\Delta} [8]\big) \,\cong \,(\Z/2\Z)^{m-j} \oplus (\Z/4\Z)^j. \]

 We will prove that the limit probability that $P_{mj}$ is satisfied given that $P_m$ is satisfied is equal to $P^{\text{Mat}}(j \, | \, m)$. In light of the last three lemmas, this is enough to prove the theorem.
 
 Given tuples $x = (p_1, \dots, p_r)$ and $x' = (p'_1, \dots, p'_r)$ in $S_{N, r, D}$, write $x \sim x'$ if 
 \begin{itemize}
 \item $p_i = p'_i$ for all $i \ne 2$ and
 \item $p_1p'_1$ is a square mod $8d_{2:r}$.
 \end{itemize}
Write $\mathscr{X}$ for the set of equivalence classes. We notice that, if a tuple $x$ satisfies $P_m$, then any $x' \sim x$ satisfies $P_m$. We can thus take $\mathscr{X}_m$ to be the subset of $\mathscr{X}$ of classes satisfying $P_m$. We are looking for the ratio
\[\frac{\sum_{\overline{x} \in \mathscr{X}_m} \left| \{x \sim \overline{x} \,:\,\, P_{mj}(x)\}\right|}{\sum_{\overline{x} \in \mathscr{X}_m} \left| \{x \sim \overline{x} \}\right|}.\]

We will prove that, for any nonempty $\overline{x} \in \mathscr{X}_m$, we have
\begin{equation}
\label{eq:equiv_c_pnt}
\frac{\left| \{x \sim \overline{x} \,:\,\, P_{mj}(x)\}\right|}{\left| \{x \sim \overline{x} \}\right|} = P^{\text{Mat}}(j \, | \, m) + O\big(e^{-\frac{1}{2}\sqrt{\ln \ln N}}\big)
\end{equation}
with the implicit constant in the O notation independent of $N$, $r$, and the choice of $x$. Taking the weighted average of these over all classes of $\mathscr{X}_m$ will then give the theorem.

Take $(p_0, p_2, \dots, p_r)$ to be a tuple in $\overline{x}$, and take $d_{2:r}$ as before. Take $L/K/\QQ$ to be the governing field associated to $(d_{2:r}, p_0)$ as per Proposition \ref{prop:8c_gov}. $K/\QQ$ is of degree $2^{r + 1}$, and $L/K$ is of degree $2^{m^2}$ since the discriminant is generic. We then get that the conductor of $L/\QQ$ is bounded by
\[N^{2^{m^2 + r + 1}}.\]
Write $p_1 \sim p_0$ if $p_1p_0$ is a square mod $8d_{2:r}$. Also write $N'$ for $Nd_{2:r}^{-1}$. Then, using Assumption \ref{ass:metab} and Proposition \ref{prop:ineff_1}, we get
\[\sum_{\substack{p_1 \sim p_0 \\ p_1 < N'}} \ln p_1 = 2^{-r-1}N' + O\big(2^r\sqrt{N'}(\ln N)^2 \big),\]
where this strong form of Chebotarev's density theorem coming from the variant of \cite[(5.109)]{Iwan04} given on the subsequent page in that book. The implied constant does not depend on $N$, $r$, or $\overline{x}$. Since $N'$ is at least $e^{e^{\sqrt{\ln \ln N}}}$, it far outpaces the error, and we easily get
\begin{equation}
\label{eq:ur1v}
\sum_{\substack{p_1 \sim p_0 \\ p_1 < N'}} \ln p_1 = 2^{-r-1}N' \cdot \left(1 +  O\left((\ln N')^{-1/2}\right)\right)
\end{equation}
 We also get
\begin{equation}
\label{eq:ur2v}
\sum_{\substack{p_1 \sim p_0 \\ p_1 < N' \\ P_{mj}(p_1, \dots, p_r)}} \ln p_1 =  2^{-r-1}P^{\text{Mat}}(j \, | \, m)N' \cdot \left(1 +  O\left((\ln N')^{-1/2}\right)\right)
\end{equation}
where the implied constant depends only $m$.

Take $P$ to be a property of primes. A classical argument \cite[(22.4.2)]{HaWr60} gives that, for $0 < \delta < 1/2$, we have
\[\sum_{P(p_1)\text{ and } p_1 < N'} 1 =  \frac{1 + O(\delta)}{\ln N'} \bigg(\sum_{P(p_1)\text{ and } p_1 < N'} \ln p_1 \bigg) + O (N'^{1-\delta}).\]
Taking $\delta = (\ln N')^{-1/2}$, we can use such estimates on the ratio of \eqref{eq:ur1v} and \eqref{eq:ur2v}, finding
\[\frac{\left| \{x \sim \overline{x} \,:\,\, P_{mj}(x)\}\right|}{\left| \{x \sim \overline{x} \}\right|} = P^{\text{Mat}}(j \, | \, m) + O\big((\ln N')^{-1/2}\big).\]
The lower bound on $N'$ then gives \eqref{eq:equiv_c_pnt}, and thus the theorem.
\end{proof}

The proof of Theorem \ref{thm:4Sel_stats} is very similar. At the beginning, we would need to take $D$ to be a multiple of the conductor of $E$, and vary $d_0$ among all divisors of $D$. We also need to derive an effective version of Proposition \ref{prop:ineff_2}. However, while the probabilities generated are different, the errors are given by the same expression, and the rest of the proof goes as above. Then Theorem \ref{thm:4Sel_stats} is also true.

\subsection{Proof of the lemmas}
\begin{proof}[Proof of Lemma \ref{lem:p1_tiny}]
Let $D$ be given. Per a result of Sathe and Selberg \cite{Selb54, Pome85}, there are positive constants $k_1, k_2$ so that, for sufficiently large $N$ and any $r$ within $(\ln \ln N )^{3/4}$ of $\ln \ln N$, we have
\[k_1  \frac{r N(\ln \ln N)^{r-1}}{\ln N} \le |S_{N, r, D}| \le k_2 \frac{r N(\ln \ln N)^{r-1}}{\ln N} .\]
Using these bounds with $S_{N, r-1, D}$, we find that the subset of $S_{N, r, D}$ with $p_1$ less than $e^{e^{\sqrt{\ln \ln N}}}$ has size less than
\[\sum_{p_1}^{e^{e^{\sqrt{\ln \ln N}}}} 2k_2\frac{r (N/p)(\ln \ln N)^{r-2}}{\ln N}\]
where the sum is over primes. The ratio of this portion to the whole is on the order of
\[(\ln \ln N)^{-1} \left(\sum_{p_1}^{e^{e^{\sqrt{\ln \ln N}}}} \frac{1}{p_1} \right)\]
which in turn is on the order of $(\ln \ln N)^{-1/2}$. This goes to zero as $N$ increases, proving the lemma.
\end{proof}

\begin{proof}[Proof of Lemma \ref{lem:generic}]
We can determine if $(d_0d_{2:r}, \,p_1)$ is generic from the Legendre symbols $\left(\frac{-1}{p_i}\right)$, $\left(\frac{2}{p_i}\right)$, and  $\left(\frac{p_i}{p_j}\right)$ for $i, j \le r$. Similarly, we can determine if $(E^{(d_{2:r})}, \,p_1)$ is generic from these Legendre symbols and the symbols $\left(\frac{q}{p_i}\right)$ for all prime divisors $q$ of the conductor of $E$.

This suggests we take the tactic of Swinnerton-Dyer and Kane. First, we will prove that non-generic discriminants and twists have zero density in the set of all possible assignments of Legendre symbols as $r$ goes to infinity. Second, we will use this result to prove that non-generic discriminants and twists have zero density in the $S_{N, r, D}$.

For fixed $d_0$ and unfixed $r$, there are on the order of $2^r$ divisors $a$ of $d_0d_{2:r}$. We will find an upper bound on the probability over assignments of Legendre symbols that a particular one of these divisors $a$ satisfies
\[(d_0d_{2:r}p_1, a)_v = (-d_0d_{2:r}p_1, a)_v = +1\]
at all rational places. So suppose $a$ has $k$ prime factors dividing $d_{2:r}$. If $p_j$ is one of these primes, we have
\begin{equation}
\label{eq:gen_cond_p}
\left(\frac{-1}{p_j}\right) =  \left(\frac{d_0d_{2:r}p_1/a}{p_j} \right) = +1.
\end{equation}
There are $2k$ Legendre symbols here. We also must have
\[\left(\frac{a}{p_i}\right) = +1\]
at all primes $p_i | d_{2:r}$ not dividing $a$, for another $r-k -1$ symbols. Added together, these make $r + k - 1$ total symbols, all independent for any nonsquare $a$. Then, over all assignments of Legendre symbols, the probability that $a$ satisfies \eqref{eq:gen_cond_p} at all $v$ is $O(2^{-r-k})$. Then the probability that any nonsquare $a$ satisfies this equation at all $v$ is
\[O\left(\sum_{k \ge 0} \binom{r}{k} 2^{-r-k} \right) = O(0.75^r)\]
by the binomial theorem. Then the density of non-generic discriminants over the set of Legendre symbol assignments approaches zero.

This argument is a rehash of the argument given by Swinnerton-Dyer in Lemmas 4 through 7 of \cite{Swin08}. It is more complicated to determine the density of Legendre symbol assignments that lead to generic elliptic curves; fortunately, this is directly what Swinnerton-Dyer proves in these lemmas. For example, with Lemma 7, Swinnerton-Dyer proves that the set of twists with nontrivial $2$-Selmer elements
\[(a_1, a_2, a_3) \quad\text{and} \quad (a'_1, a'_2, a'_3)\quad \text{with}\quad a_1 = a'_2\]
has density zero over the set of Legendre symbol assignments.

Given a discriminant $\Delta = dp_1$ of an imaginary quadratic field, take $V_{ng}(d, p_1)$ to be the space of $a$ dividing $d$ so
\[(\Delta, a)_v = (-\Delta, a)_v = +1\]
at all places. Here, $ng$ stands for nongeneric. $V_{ng}(d, p_1)$ always corresponds a subgroup of $2(\text{Cl}\, K_{\Delta}[4])$, and over the distribution of Legendre symbols, the proportion of discriminants with $V_{ng}(d_0d_{2:r}, p_1)$ nontrivial approaches zero. Since the limit distribution of $2(\text{Cl}\, K_{\Delta}[4])$ over all possible Legendre symbol assignments has finite moments, we find that the moments of $|V_{ng}(d_0d_{2:r}, p_1)|$ over the distribution of Legendre symbols all approach one as $r$ increases.

Similarly, take $W_{ng}(E^{(d)}, p_1)$ to be the subset of $2$-Selmer elements 
\[F = (a_1, a_2, a_3) \in \SelT(E^{(dp_1)})\]
that are unramified at $p_1$ and so that there is some other Selmer element $F' = (a'_1, a'_2, a'_3)$, also unramified at $p_1$, with $a'_1 = a_2$. Similarly, take $W'_{ng}(E^{(d)}, p_1)$ to be the subset of $2$-Selmer elements in the form $(1, a_2, a_2)$.  Following the same argument as for discriminants, we find that the moments of $|W_{ng}(E^{(d_{2:r})}, p_1)|$ and $|W'_{ng}(E^{(d_{2:r})}, p_1)|$  over the distribution of Legendre symbols all approach one as $r$ increase.

The only distribution that agrees with any of these moments is the distribution where $100\%$ of the spaces $V_{ng}$, $W_{ng}$, $W'_{ng}$ are trivial. All we need to now is prove that some non-trivial moment calculated over the distribution of Legendre symbols is equal to the corresponding moment calculated over the average of the $S_{N, r, D}$.

For elliptic curves, we will actually prove that the average of
\begin{equation}
\label{eq:wng_wpng}
|W_{ng}(E^{(d)}, p_1))| \cdot |W'_{ng}(E^{(d)}, p_1)|
\end{equation}
is one over $S_{N, r, D}$, which will be enough to show the distribution of both groups is trivial. It is immediate from the Cauchy-Schwarz inequality that the average of this product over the distribution of Legendre symbols is trivial.

Our proof comes directly from the proof of Proposition 18 in Kane's article \cite{Kane13}, and we explain the modifications needed for his argument in the context of this proof. In this proof, Kane writes the $2$-Selmer group of $E^{(dp_1)}$ as the intersection of two spaces $U, W$ in a larger space $V$. There is a natural nondegenerate alternating pairing 
\[\langle\,\,\,,\,\,\,\rangle: V \times V \rightarrow \mathbb{F}_2.\]
To be concrete, $U$ consists of tuples
\[(a_1, a_2, a_3) \in (\QmodSq\big)^{\oplus 3}\,\,\, \text{with }\,\,a_1a_2a_3 = 1\]
with no $a_i$ divisible by a prime outside $dp_1$ or twice the conductor of $E$. Kane's computation of the second moment starts with
\[\big|\SelT(E^{(dp_1)})\big|^2 = \left|(U \cap W)^{\oplus 2}\right| = \frac{1}{|W|^2}\sum_{\substack{w_1, w_2 \in W\\u_1, u_2 \in U}}  (-1)^{\langle w_1, u_1 \rangle + \langle w_2, u_2 \rangle}.\]
Take $U^2_{ng}$ to consist of tuples 
\[\big((a_1, a_2, a_3), (a'_1, a'_2, a'_3)\big) \in U^{\oplus 2}\]
with $a_2 = a'_1$ and with all $a_i$ and $a'_i$ not divisible by $p_1$. We see that
\[U^2_{ng} \cap W^{\oplus2} = U^2_ng \cap \big(\SelT(E^{(dp_1)})\big)^{\oplus 2},\]
and from this we can find that 
\[|U^2_{ng} \cap W^{\oplus2}| = |W_{ng}(E^{(d)}, p_1))| \cdot |W'_{ng}(E^{(d)}, p_1)|.\]
Then
\[ |W_{ng}(E^{(d)}, p_1))| \cdot |W'_{ng}(E^{(d)}, p_1)| = \frac{1}{|W|^2}\sum_{\substack{w_1, w_2 \in W\\(u_1, u_2) \in U^2_{ng}}}  (-1)^{\langle w_1, u_1 \rangle + \langle w_2, u_2 \rangle}.\]
We can use the rest of Kane's argument verbatim to find the average of this sum over $S_{N, r, D}$. Then the limit of the average of \eqref{eq:wng_wpng} over assignments of Legendre symbols equals the limit of its average over the $S_{N, r, D}$. This gives the lemma.

The translation from distributions over Legendre symbols to distributions over $S_{N, r, D}$ follows similarly for non-generic discriminants once we have phrased the calculation of $4$-class ranks in a form that Kane can deal with. This is done in \cite{Smith16}, and we subsequently find that the limit density of non-generic discriminants over the $S_{N, r, D}$ is zero. The lemma follows.
\end{proof}

\bibliography{references}{}
\bibliographystyle{amsplain}

\end{document}